\documentclass[11pt]{amsart}
\usepackage{amsmath}
\usepackage[margin=2.0cm]{geometry}
\usepackage{enumitem}

\usepackage{tikz}

\usepackage{tikz}

\theoremstyle{plain}
\newtheorem{theorem}{Theorem}[section]

\theoremstyle{definition}

\theoremstyle{remark}
\newtheorem{remark}[theorem]{Remark}

\title[Cubic surfaces as Pfaffians]{Cubic surfaces as Pfaffians}
\author{David Oscari}
\address{FaMAFyC and CIEM, Universidad Nacional de C\'ordoba, C\'ordoba, Argentina}
\email{oscari@famaf.unc.edu.ar}

\begin{document}

 
\maketitle

\begin{abstract}
We prove that every cubic surface in $\mathbb{C}[x,y,z,t]$ is Pfaffian. A constructive proof is given.
\end{abstract}

\section{Introduction}
 
 Let $f(x,y,z,t)$ be a homogeneous polynomial of degree  three, with coefficients in $\mathbb{C}$:
\newcommand{\coeff}{\Theta} 
\begin{align}
f(x,y,z,t) & =\coeff_1\,x^3+\coeff_2\,y^3+\coeff_3\,z^3+\coeff_4\,t^3+ \nonumber  \\
  & \hspace{8.65ex}
 +\coeff_5\,x^2y+\coeff_6\,xy^2
 +\coeff_7\,x^2z+\coeff_8\,xz^2
 +\coeff_9\,x^2t+\coeff_{10}\,xt^2
   \nonumber \\
   & \hspace{8.65ex} 
 +\coeff_{11}\,y^2z+\coeff_{12}\,yz^2
 +\coeff_{13}\,y^2t+\coeff_{14}\,yt^2
 +\coeff_{15}\,z^2t+\coeff_{16}\,zt^2
\label{eq:1}    \\
   & \hspace{39.38ex}
 +\coeff_{17}\,xyz+\coeff_{18}\,xyt+\coeff_{19}\,xzt+\coeff_{20}\,yzt\,.  \nonumber
\end{align}
It is said that $f$ is {\it Pfaffian} if there exists a matrix 
$$
M=xA_0+yA_1+zA_2+tA_3\,,
$$
where $A_0,A_1,A_2,A_3$ are $6\times6$ skew-symmetric matrices, with entries in $\mathbb{C}$, such that
$$
\det\,M=f(x,y,z,t)^2.    
$$
If the sign of $M$ is chosen appropriately, then  $M$ is called a {\it Pfaffian representation} of $f$. 

Although it is known that every cubic surface is Pfaffian (\cite{B}, \cite{FM}), few algorithms are known whose input is a cubic surface and whose output is a {\it explicit} Pfaffian representation. In \cite{H} and \cite{T} algorithms are provided that produce such representations, from a cubic surface in $\mathbb{K}[x,y,z,t]$, where $\mathbb{K}$ is a field. 
 
In this note we give a constructive proof of the following theorem.
\begin{theorem}\label{main}
Every cubic surface in $\mathbb{C}[x,y,z,t]$ is Pfaffian.
\end{theorem}

When $y=0$, we obtain a ternary cubic:
\begin{multline}\label{cubica ternaria}
f(x,0,z,t) =\coeff_1\,x^3+\coeff_3\,z^3+\coeff_4\,t^3
                +\coeff_7\,x^2z+\coeff_8\,xz^2+\coeff_9\,x^2t+\coeff_{10}\,xt^2
                +\coeff_{15}\,z^2t+\coeff_{16}\,zt^2+\coeff_{19}\,xzt\,. 
\end{multline}

In \cite{KM}, I. Kogan and M. Moreno Maza provided computationally efficient algorithm that determines, up to a linear change of variables, the canonical form of an arbitrary ternary
cubic, and {\it explicitly} computes a corresponding matrix of change.

By Theorem 5 of ibid., there exists  a linear change of variables, say 
$
A_0=\left[
\begin{smallmatrix}
a_{11} & a_{12} & a_{13} \\
a_{21} & a_{22} & a_{23} \\
a_{31} & a_{32} & a_{33} 
\end{smallmatrix}
\right]\in \textsf{GL}(3,\mathbb{C})\,,
$
such that $f(x,0,z,t)$ can be transformed by $A_0$ to one  of the following canonical forms:
\begin{enumerate}[label=(\Roman*)]
    \item\label{irreducible} Irreducible (ternary) cubic
    \begin{align*} 
           x^3+\alpha xz^2+z^3-t^2z, && x^3+xz^2-t^2z, &&  x^3+ z^3-t^2z, && x^3-t^2z, && x^3+x^2z-t^2z,
    \end{align*}
    where $\alpha\not=0$ and $\alpha^3\not=27/4$.
    \item\label{reducible} Reducible (ternary) cubic
    \begin{align*}
           z(x^2+tz), && z(x^2+t^2+z^2), &&  xtz, && xt(x+t), && x^2t, && x^3. 
    \end{align*} 
\end{enumerate}

\section{Case irreducible when $y=0$}

For our purposes, if $f(x,0,z,t)$ is irreducible, then each of its canonical forms in \ref{irreducible} can be written as
\newcommand{\coef}{\Lambda} 
\begin{equation}\label{forma canonica}
x^3+\coef_{8}\,xz^2+\coef_3\,z^3-t^2z+\coef_7\,x^2z,\qquad \text{where }\coef_3,\coef_8\in\mathbb{C}.    
\end{equation}
Now we apply the linear change of variables induced by 
$
\left[
\begin{smallmatrix}
a_{11} & a_{12} & a_{13} \\
a_{21} & a_{22} & a_{23} \\
a_{31} & a_{32} & a_{33} 
\end{smallmatrix}
\right]
$ to the cubic surface (\ref{eq:1}):
\begin{align}
f\left(
(x,y,z,t)\cdot 
\left[
\begin{smallmatrix}
a_{11} & 0 & a_{12} & a_{13} \\
     0 & 1 &  0  & 0 \\
a_{21} & 0 & a_{22} & a_{23} \\
a_{31} & 0 &a_{32} & a_{33} 
\end{smallmatrix}
\right]
\right) & = {\Bigg(}\text{canonical form (\ref{forma canonica}) of }f(x,0,z,t){\Bigg)}+{\Bigg(}\text{monomials that contain }y{\Bigg)} \nonumber  \\
    & ={\Bigg(}x^3+\coef_{8}\,xz^2+\coef_3\,z^3-t^2z+\coef_7\,x^2z{\Bigg)} +{\Bigg(}\coef_2\,y^3+\coef_5\,x^2y+\coef_6\,xy^2
\nonumber  \\
   & \hspace{9ex}
  +\coef_{11}\,y^2z 
 +\coef_{12}\,yz^2
 +\coef_{13}\,y^2t +\coef_{17}\,xyz+\coef_{18}\,xyt+\coef_{20}\,yzt{\Bigg)}\,.
   \nonumber 
\end{align}

Let $B_0,B_1,B_2$ and $B_3$ be $6\times6$ skew-symmetric matrices defined by
$$
B_0=\left[ \begin {smallmatrix}
0&1&0&0&0&-\Lambda_{7}\\ 
-1&0&0&0&0&-\Lambda_{5}\\ 
0&0&0&-1&0&0\\ 
0&0&1&0&0&0\\ 
0&0&0&0&0&-1\\ 
\Lambda_{7}&\Lambda_{5}&0&0&1&0
\end {smallmatrix} \right], 
\hspace{2ex}
B_1=\left[ \begin {smallmatrix} 
0&0&\Lambda_{{11}}&0&-1&\Lambda_{{2}}-\Lambda_{{17}}\\ 
0&0&\Lambda_{{2}}&0&0&-\Lambda_{{6}}\\ 
-\Lambda_{{11}}&-\Lambda_{{2}}&0&0&0&\Lambda_{{12}}\\ 
0&0&0&0&0&1\\ 
1&0&0&0&0&0\\ 
-\Lambda_{{2}}+\Lambda_{{17}}&\Lambda_{{6}}&-\Lambda_{{12}}&-1&0&0
\end {smallmatrix} \right], 
\hspace{2ex}
B_2=\left[ \begin {smallmatrix} 
0&0&0&-1&0&-\Lambda_{{8}}\\ 
0&0&0&0&1&0\\ 
0&0&0&0&0&\Lambda_{{3}}\\ 
1&0&0&0&0&0\\ 
0&-1&0&0&0&0\\ 
\Lambda_{{8}}&0&-\Lambda_{{3}}&0&0&0
\end {smallmatrix} \right]\,, 
$$
$$ 
B_3=\left[ \begin {smallmatrix} 
0&0&\Lambda_{{20}}-d_{{11}}\Lambda_{{6}}&0&0&\Lambda_{{13}}+d_{{11}}\Lambda_{{5}}\\ 
0&0&\Lambda_{{13}}&0&0&-\Lambda_{{18}}-d_{{11}}\\ 
-\Lambda_{{20}}+d_{{11}}\Lambda_{{6}}&-\Lambda_{{13}}&0&0&d_{{11}}&0\\ 
0&0&0&0&0&0\\ 
0&0&-d_{{11}}&0&0&0\\ 
-\Lambda_{{13}}-d_{{11}}\Lambda_{{5}}&\Lambda_{{18}}+d_{{11}}&0&0&0&0
\end {smallmatrix}
\right]\,, 
\quad \text{where } \textstyle d_{11}=\frac{-\Lambda_{18}+\sqrt{\Lambda_{18}^2-4}}{2}\,.
 $$
And let $M_0$ be the matrix 
\begin{equation}\label{pfaffian rep}
M_0{\Big(}x,y,z,t{\Big)}:=xB_0+yB_1+zB_2+tB_3\,.    
\end{equation}
\begin{remark}\label{det M_0 = f^2}
The matrix $M_0$ defined by (\ref{pfaffian rep}) is a Pfaffian representation of 
$f\left(
(x,y,z,t)\cdot 
\left[
\begin{smallmatrix}
a_{11} & 0 & a_{12} & a_{13} \\
     0 & 1 &  0  & 0 \\
a_{21} & 0 & a_{22} & a_{23} \\
a_{31} & 0 &a_{32} & a_{33} 
\end{smallmatrix}
\right]
\right)$.
\end{remark}

\begin{proof}[Proof of the Theorem \ref{main}] Let $f(x,y,z,t)$ be a arbitrary cubic surface defined by (\ref{eq:1}). When $y=0$, we obtain a ternary cubic $f(x,0,z,t)$  as in (\ref{cubica ternaria}).

If $f(x,0,z,t)$ is irreducible, then, by the Remark \ref{det M_0 = f^2},
$$
\det\,M_0{\Big(}x,y,z,t{\Big)}
=
f\left(
(x,y,z,t)\cdot 
\left[
\begin{smallmatrix}
a_{11} & 0 & a_{12} & a_{13} \\
     0 & 1 &  0  & 0 \\
a_{21} & 0 & a_{22} & a_{23} \\
a_{31} & 0 &a_{32} & a_{33} 
\end{smallmatrix}
\right]
\right)^2\,.
$$
Therefore,
$$
\det\,M_0\left(
(x,y,z,t)\cdot 
\left[
\begin{smallmatrix}
a_{11} & 0 & a_{12} & a_{13} \\
     0 & 1 &  0  & 0 \\
a_{21} & 0 & a_{22} & a_{23} \\
a_{31} & 0 &a_{32} & a_{33} 
\end{smallmatrix}
\right]^{-1}
\right)
=
f\left(
(x,y,z,t)\cdot 
\left[
\begin{smallmatrix}
a_{11} & 0 & a_{12} & a_{13} \\
     0 & 1 &  0  & 0 \\
a_{21} & 0 & a_{22} & a_{23} \\
a_{31} & 0 &a_{32} & a_{33} 
\end{smallmatrix}
\right]\cdot
\left[
\begin{smallmatrix}
a_{11} & 0 & a_{12} & a_{13} \\
     0 & 1 &  0  & 0 \\
a_{21} & 0 & a_{22} & a_{23} \\
a_{31} & 0 &a_{32} & a_{33} 
\end{smallmatrix}
\right]^{-1}
\right)^2
=
f(x,y,z,t)^2\,.
$$

For the case $f(x,0,z,t)$ reducible, there are at most six sub-cases, one for each canonical form in \ref{reducible}. Each of those sub-cases is trivial.
\end{proof}

\bibliographystyle{plain}

\def\cprime{$'$}

\end{document}